\documentclass[12pt]{amsart}
\usepackage{amscd}
\usepackage{amsmath}
\usepackage{amssymb}
\usepackage{units}
\usepackage[all]{xy}

\usepackage{algorithm}
\usepackage{algorithmic}

\def\frk{\frak}               

\def\mm{{\frk m}}

\def\Phi{{\frk n}}
\def\Phi{{\frk N}}
\def\opn#1#2{\def#1{\operatorname{#2}}} 
\opn\chara{char} \opn\length{\ell} \opn\pd{pd} \opn\rk{rk}
\opn\projdim{proj\,dim} \opn\injdim{inj\,dim} \opn\rank{rank}
\opn\depth{depth} \opn\sdepth{sdepth} \opn\fdepth{fdepth}
\opn\grade{grade} \opn\height{height} \opn\embdim{emb\,dim}
\opn\codim{codim}  \opn\min{min} \opn\max{max}

\opn\Tr{Tr} \opn\bigrank{big\,rank}
\opn\superheight{superheight}\opn\lcm{lcm}
\opn\trdeg{tr\,deg}
\opn\reg{reg} \opn\lreg{lreg} \opn\ini{in} \opn\lpd{lpd}
\opn\size{size}
\opn\div{div} \opn\Div{Div} \opn\cl{cl} \opn\Cl{Cl}
\opn\Spec{Spec} \opn\Supp{Supp} \opn\supp{supp} \opn\Sing{Sing}
\opn\Ass{Ass} \opn\Min{Min}
\opn\Ann{Ann} \opn\Rad{Rad} \opn\Soc{Soc}
\opn\Im{Im} \opn\Ker{Ker} \opn\Coker{Coker} \opn\Am{Am}
\opn\Hom{Hom} \opn\Tor{Tor} \opn\Ext{Ext} \opn\End{End}
\opn\Aut{Aut} \opn\id{id}  \opn\deg{deg}

\opn\nat{nat}
\opn\pff{pf}
\opn\Pf{Pf} \opn\GL{GL} \opn\SL{SL} \opn\mod{mod} \opn\ord{ord}
\opn\Gin{Gin} \opn\Hilb{Hilb}
\opn\aff{aff} \opn\con{conv} \opn\relint{relint} \opn\st{st}
\opn\lk{lk} \opn\cn{cn} \opn\core{core} \opn\vol{vol}
\opn\link{link} \opn\star{star}
\opn\gr{gr}

\def\pot#1#2{#1[\kern-0.28ex[#2]\kern-0.28ex]}

\opn\dirlim{\underrightarrow{\lim}}
\opn\inivlim{\underleftarrow{\lim}}
\let\to=\rightarrow

\def\Implies{\ifmmode\Longrightarrow \else
        \unskip${}\Longrightarrow{}$\ignorespaces\fi}
\def\implies{\ifmmode\Rightarrow \else
        \unskip${}\Rightarrow{}$\ignorespaces\fi}
\def\iff{\ifmmode\Longleftrightarrow \else
        \unskip${}\Longleftrightarrow{}$\ignorespaces\fi}

\let\:=\colon
\newtheorem{Theorem}{Theorem}[]
\newtheorem{Lemma}[Theorem]{Lemma}

\newtheorem{Proposition}[Theorem]{Proposition}
\newtheorem{Remark}[Theorem]{Remark}

\let\epsilon\varepsilon
\let\phi=\varphi
\let\kappa=\varkappa
\def\qed{\ifhmode\textqed\fi
      \ifmmode\ifinner\quad\qedsymbol\else\dispqed\fi\fi}
\def\textqed{\unskip\nobreak\penalty50
       \hskip2em\hbox{}\nobreak\hfil\qedsymbol
       \parfillskip=0pt \finalhyphendemerits=0}
\def\dispqed{\rlap{\qquad\qedsymbol}}

\opn\dis{dis}
\def\pnt{{\raise0.5mm\hbox{\large\bf.}}}

\opn\Lex{Lex}

\begin{document}

\title{ Simple General Neron Desingularization in  local $\mathbb{Q}$-algebras.}

\author{ Dorin Popescu }
\thanks{}

\address{Dorin Popescu, Simion Stoilow Institute of Mathematics of the Romanian Academy, Research unit 5,
University of Bucharest, P.O.Box 1-764, Bucharest 014700, Romania}
\email{dorin.popescu@imar.ro}

\begin{abstract} We give here a  easier proof of the so-called General Neron Desingularization in the frame of local algebras.

\end{abstract}

\maketitle

\vskip 0.5 cm

\section*{Introduction}

A ring morphism $u:A\to A'$ of Noetherian rings has  {\em regular fibers} if for all prime ideals $p\in \Spec A$ the ring $A'/pA'$ is a regular  ring.
It has {\em geometrically regular fibers}  if for all prime ideals $p\in \Spec A$ and all finite field extensions $K$ of the fraction field of $A/p$ the ring  $K\otimes_{A/p} A'/pA'$ is regular.
A flat morphism of Noetherian rings $u$ is {\em regular} if its fibers are geometrically regular. If $u$ is regular of finite type then $u$ is called {\em smooth}.

The following theorem extends N\'eron's desingularization (see \cite{N}, \cite{KAP}) and it was used to solve different problems concerning  projective modules over regular rings, or in Artin Approximation Theory (see
\cite{A}, \cite{P}, \cite{P1}, \cite{P2}, \cite{P3}, \cite{S}, \cite{R}, \cite{H}).

\begin{Theorem} (General N\'eron Desingularization, Popescu \cite{P}, \cite{P'}, \cite{P1}, Andr\'e \cite{An}, Swan \cite{S})\label{gnd}  Let $u:A\to A'$ be a  regular morphism of Noetherian rings and $B$ an  $A$-algebra of finite type. Then  any $A$-morphism $v:B\to A'$   factors through a smooth $A$-algebra $C$, that is $v$ is a composite $A$-morphism $B\to C\to A'$.
\end{Theorem}

 Other authors gave  similar proofs (see \cite{Sp}, \cite{Q}, \cite{MR}) but most of them were not constructive. In \cite{AP}, \cite{PP}, \cite{PP1}, 
 \cite{KK},    
 \cite{KPP},
 \cite{KAP}, \cite{ZKPP}, there exist some constructive proofs with algorithms in some special cases. We mention that all these proofs 
are simpler when we deal with $\mathbb{Q}$-algebras because in this case all residue field extensions are separable and it is easier to handle the regularity of the morphisms.

  Let $A$ be a Noetherian ring, $B=A[Y]/I$, $Y=(Y_1,\ldots,Y_n)$. If $f=(f_1,\ldots,f_r)$, $r\leq n$ is a system of polynomials from $I$ then we can define the ideal $\Delta_f$ generated by all $r\times r$-minors of the Jacobian matrix $(\partial f_i/\partial Y_j)$.   After Elkik \cite{El} let $H_{B/A}$ be the radical of the ideal $\sum_f ((f):I)\Delta_fB$, where the sum is taken over all systems of polynomials $f$ from $I$ with $r\leq n$.
 $H_{B/A}$ defines the non smooth locus of $B$ over $A$.
  $B$ is {\em standard smooth} over $A$ if  there exists  $f$ in $I$ as above such that $B= ((f):I)\Delta_fB$.

  In \cite{ZKPP}, it is given  the following constructive theorem.
\vskip 0.3 cm
 \begin{Theorem} \label{m}  Let $A$ and $A'$ be Noetherian local rings, $u:A\to A'$ be a regular morphism and $B$ an $A$-algebra of finite type. Suppose that $A'$ is Henselian and the maximal ideal $\mm$ of $A$ generates the maximal ideal of $A'$. Any $A$-morphism $v:B\to A'$  such that $v(H_{B/A}A')$ is $\mm A'$-primary ideal factors through a standard smooth $A$-algebra $B'$.
 \end{Theorem}

The proof was based on 
the following proposition (see \cite[Proposition 3]{ZKPP}, where the conditions "$A'$ is Henselian, and $A,A'$ have  dimensions $m$ are superfluous" and we remove them here).

 \begin{Proposition}\label{p}
 Let $A$ and $A'$ be Noetherian local rings  and $u:A\to A'$ be a regular morphism.  Let $B=A[Y]/I$, $Y=(Y_1,\ldots,Y_n)$, $f=(f_1,\ldots,f_r)$, $r\leq n$ be a system of polynomials from $I$ as above, $(M_j)_{j\in [l]}$  some $r\times r$-minors  \footnote{ We use the notation $[l]=\{1,\ldots,l\}$.}  of the Jacobian matrix $(\partial f_i/\partial Y_{j'})$,  $(N_j)_{j \in [l]} \in  ((f):I)$ and set $P:=\sum _{j=1}^l N_jM_j$. Let  $v:B\to A'$ be an $A$-morphism. Suppose that
\begin{enumerate}
 \item{} there exists an element $d\in A$ such that $d\equiv P $ modulo $I$ and

 \item{} there exist a smooth $A$-algebra $D$ and an $A$-morphism $\omega:D\to A'$ such that $\Im v\subset \Im\omega +d^{2e+1}A'$ and for ${\bar A}=A/(d^{2e+1})$ (defining e by $(0:_Ad^e)=(0:_Ad^{e+1})$) the map $\bar{v}={\bar A} {\otimes}_{A} v: \bar{B}=B/d^{2e+1}B \to \bar{A}'=A'/d^{2e+1}A'$ factors through $\bar{D}=D/d^{2e+1}D$.
\end{enumerate}
 Then there exists a $B$-algebra $B'$ which is standard smooth over $A$ such that $v$ factors through $B'$.
 \end{Proposition}

It is the purpose of our paper to provide a easier proof (see Theorem \ref{gnd0}) of Theorem \ref{gnd} in the local $\mathbb{Q}$-algebras following the ideas from Theorem \ref{m} and Proposition \ref{p}. Here $\mathbb{Q}$
denotes the field of rational numbers.
The proof is  not constructive (see Remark \ref{r}), but Proposition \ref{p2} shows  a case when this is so. 
We mention that the proof of Theorem \ref{gnd0} is easier, because mainly we had to apply only  Proposition \ref{p1}, which concentrates in short the ideas of the previous proofs from \cite{P0},
\cite{P1} in the frame of $\mathbb{Q}$-algebras. 

\section{A simple proof of the General Neron Desingularization}
\vskip 0.3 cm

Let $q\in \Spec A'$ be  a minimal prime associated ideal of $h_B=\sqrt{v(H_{B/A})A'}$. After \cite{S} we say that $A\to B\to A'\supset q$ is {\em resolvable} if there exists an $A$-algebra of finite type $C$ such that $v$ factors through $C$, let us say $v=\beta t$, $t:B\to C$, $\beta:C\to A'$ and $h_B\subset h_C=\sqrt{\beta(H_{C/A})A'}\not \subset q$.

The following lemma is in fact \cite[Lemma 6.6]{P1} which unifies \cite[Lemmas 12.2, 12.3]{S}. We give here  its proof in sketch for the sake of completeness.

\begin{Lemma} \label{l0} Suppose that $\height q=0$ and $A_p\to B_p= A_p\otimes_AB\otimes A'_q\supset qA'_q$,  $p=u^{-1}(q)$ is resolvable. Then $A\to B\to A'\supset q$ is resolvable too.  
\end{Lemma}
\begin{proof} By hypothesis there exists an $A_p$-algebra $D$ such that $v_p=A_p\otimes_A v$ factors through $D$, let us say $v_p$ is the composite map $B_p\to D\xrightarrow{\beta_p} A'_q$  and $\beta_p(H_{D/A_p})=A'_q$. Moreover we may take $D$ to be smooth over $A_p$ using for instance the easy lemma \cite[Lemma 2.4]{P0}

Let $D=B_p[Z]/(g)$, $Z=(Z_1,\ldots,Z_s)$, $g=(g_1,\ldots,g_e)\in B^e$ and $\beta_p$ be given by $Z\to z/z_0$ for some $z\in A'^s$, $z_0\in A'\setminus q$
Take the homogeneous polynomials $G(Z,Z_0)=Z_0^c g(Z/Z_0)$ for some $c>>1$. We have 
$G(z,z_0)=0$ in $A'_q$ and so $aG(z,z_0)=0$ in $A'$ for some $a\in A'\setminus q$. Changing $z,z_0$ by $az,az_0$ we may assume that $G(z,z_0)=0$ in $A'$. Let $C=B[Z,Z_0]/(G)$ and $\beta:C\to A'$ be the extension of $v$ by 
$(Z,Z_0)\to (z,z_0)$. We get a $B_p$-isomorphism $\nu:D[Z_0,Z_0^{-1}]\to (B_p\otimes_BC)_{Z_0} $ by $(Z,Z_0)\to (Z/Z_0,Z_0)$ and the map $C_{p,Z_0}\to A'_q$ induced by $\beta$, factors through $\nu$. We may change $D$ by the smooth $A_p$-algebra $D[Z_0,Z_0^{-1}]$, resp. $g$ by $G$, $C$ by $B[Z]/(g)$.
We have $h_C=\sqrt{\beta(H_{C/A})A'}\not \subset q$, let us say $\beta(t)\not\in q$ for some $t\in H_{C/A}$.

If $h_B\subset h_C$ there exists nothing to show. Otherwise, $h_B^k=0$ in $A'_q$, that is $wh_B^k=0$ in $A'$ for some $w\in A'\setminus q$. Let $b_1,\ldots,b_s$ be a system of generators of $h_B$ and set 
$f_j=g_j-\sum_iU_iV_{ij}$, for some new variables $U=(U_1,\ldots, U_s)$, $V=(V_{ij})_{i\in [s], j\in [e]}$. 
Then
$$E=B[Z,U,V,W]/(f,WU_1,\ldots,WU_s)$$
and $\delta:E\to A'$ extending $v$ by $U\to b$, $V\to 0$, $W\to w$ resolves
$A\to B\to A'\supset q$. Indeed, $E_{U_i}$ is a polynomial algebra over $B[U_i,U_i^{-1}]$, so smooth over $B$. Thus $E_{b_iU_i}$ is smooth over $A$ and $b_iU_i\in H_{E/A}$, that is $h_B\subset h_E=\sqrt{\delta(H_{E/A})A'}$. Note that $E_W\cong B[Z,V,W,W^{-1}]/(g)\cong C[V,W,W^{-1}]$
and so $E_{tW}$ is smooth over $A$. Hence $\beta(t)w\in h_E\setminus q$. 
\hfill\  \end{proof}

The  proof of Proposition \ref{p} works for the extension given below (its idea  comes from \cite[Proposition 7.1]{P0} and \cite[Lemma 7.2]{P1}).

 \begin{Proposition}\label{p1}
 Let $A$ and $A'$ be Noetherian  rings  and $u:A\to A'$ be a flat morphism. Suppose that $A'$ is local. Let $B=A[Y]/I$, $Y=(Y_1,\ldots,Y_n)$, $f=(f_1,\ldots,f_r)$, $r\leq n$ be a system of polynomials from $I$ as above, $(M_j)_{j\in [l]}$  some $r\times r$-minors    of the Jacobian matrix $(\partial f_i/\partial Y_{j'})$,  $(N_j)_{j \in [l]} \in  ((f):I)$ and set $P:=\sum _{j=1}^l N_jM_j$. Let  $v:B\to A'$ be an $A$-morphism. Suppose that
\begin{enumerate}
 \item{} there exists an element $d\in A$ such that $d\equiv P $ modulo $I$ and

 \item{} there exist a flat  $A$-algebra $D$ of finite type and an $A$-morphism $\omega:D\to A'$ such that $\Im v\subset \Im\omega +d^{2e+1}A'$ and for ${\bar A}=A/(d^{2e+1})$ (defining e by $(0:_Ad^e)=(0:_Ad^{e+1})$) the map $\bar{v}={\bar A} {\otimes}_{A} v: \bar{B}=B/d^{2e+1}B \to \bar{A}'=A'/d^{2e+1}A'$ factors through $\bar{D}=D/d^{2e+1}D$.
\end{enumerate}
 Then there exists a smooth $D$-algebra $B'$  such that $v$ factors through $B'$, let us say $v$ is the composite map $B\to B'\xrightarrow{w} A'$ and  $h_D=\sqrt{\omega(H_{D/A})A'}\subset h_{B'}=\sqrt{w(H_{B'/A})A'}$.
 \end{Proposition}

 \begin{proof} (Sketch after \cite[Proposition 3]{ZKPP})
 Let $\delta:B\otimes_AD\cong D[Y]/ID[Y]\to A'$ be the $A$-morphism given by $b\otimes \lambda\to v(b)\omega(\lambda)$.
We  show that
 $\delta$ factors through a special  $B\otimes_AD$-algebra $E$ of finite type.

By hypothesis $\bar v$ factors through 
a map $\bar B\to \bar D$ given, let us say, by $Y\to y'+d^{2e+1}D$, $y'\in D^n$. Thus $I(y')\equiv 0$ modulo $d^{2e+1}D$ and ${\tilde y}:=\omega (y')\equiv v(Y)=: y$ modulo $d^{2e+1}A'^n$, let us say 
 $y-\tilde y=d^{e+1}\nu$ for some $\nu\in d^eA'^n$.

For $j\in [q]$ we may complete the matrix $(\partial f_i/\partial Y_{j'})_{i\in [r],{j'}\in [n]}$  with some $(n-r)$ rows from $0,\ 1$ to get a square matrix $H_j$ with
 $M_j=\det H_j$. Since $d\equiv P$ modulo $I$ it follows that  and so  $P(y')=ds$ for some $s\in D$ with $s\equiv 1$ modulo $d$.
 Let $G'_j$ be the adjoint matrix of $H_j$ and $G_j=N_jG'_j$. We have
$G_jH_j=H_jG_j=M_jN_j\mbox{Id}_n$
and 
$$ds\mbox{Id}_n=P(y')\mbox{Id}_n=\sum_{j=1}^{q}G_j(y')H_j(y').$$
We obtain
 \begin{equation}\label{identity1}(\partial f/\partial Y){G_j}=(M_jN_j\mbox{Id}_r|0).\end{equation}
 ${t_j}:=H_j(y')\nu\in d^eA'^n$
satisfies
$${G}_j(y'){t_j}=M_j(y')N_j(y')\nu=d{s}\nu$$
 and 
 $${s}(y-\tilde y)=d^e\sum_{j=1}^{q}\omega({G}_j(y')){t_j}.$$
 Let
 \begin{equation}\label{def of h1}{h}={s}(Y-y')-d^e\sum_{j=1}^{q}{G}_j(y'){T_j},\end{equation}
 where  ${T_j}=({T}_1,\ldots,{T}_r, T_{j,r+1}, \ldots , T_{j,n})$ are new variables. The kernel of the map
${\phi}:D[Y,{T}]\to A'$ given by $Y\to y$, ${T_j}\to {t_j}$ contains ${h}$. Since
$${s}(Y-y')\equiv d^e\sum_{j=1}^{q}{G}_j(y')T_j\ \mbox{modulo}\ {h}$$
and
$$f(Y)-f(y')\equiv \sum_{j'}(\partial f/\partial Y_{j'})((y') (Y_{j'}-y'_{j'})$$
modulo higher order terms in $Y_{j'}-y'_{j'}$, by Taylor's formula we see that for $p=\max_i \deg f_i$ we have
\begin{equation}\label{def of Q1}{s}^pf(Y)-{s}^pf(y')\equiv  \sum_{j'}{s}^{p-1}d^e(\partial f/\partial Y_{j'})(y')\sum_{j=1}^{q} {G}_{jj'}(y'){T}_{jj'}+d^{2e}{Q}\end{equation}
modulo $h$ where ${Q}\in {T}^2 D[{T}]^r$.   We have $f(y')=d^{e+1}{b}$ for some ${b}\in d^eD^r$. Then
\begin{equation}\label{def of g1}{g}_i={s}^p{b}_i+{s}^p{T}_i+d^{e-1}{Q}_i, \qquad i\in [r] \end{equation}  is in the kernel of $\phi$ as in the proof of \cite[Proposition 3]{ZKPP}.

 Set ${E}=D[Y,{T}]/(I,{g},{h})$ and let  ${\psi}:{E}\to A'$ be the map induced by $\phi$. Clearly, $v$ factors through $\psi$ because $v$ is the composed map $B\to B\otimes_AD\cong D[Y]/I\to {E}\xrightarrow{{\psi}} A'$. Note that the $r\times r$-minor  ${s}'$ of $(\partial {g}/ \partial {T})$ given by the first  $r$-variables ${T}$ is from ${s}^{rp}+({T})\subset 1+(d,{T})$ because $Q\in ({T})^2$. Then $V:=(D[Y,{T}]/({h},{g}))_{{s}{s}'}$ is smooth over $D$.

As in the proof of \cite[Proposition 3]{ZKPP} we see that $I\subset ({h},{g})D[Y,{T}]_{{s}{s}'{s}''}$ for some other ${s}''\in 1+(d,{T})D[Y,{T}]$. Indeed, we have $ (({h},{g})D[Y,{T}]_{s}):I$ contains $P$ and so 
$P(y'+{s}^{-1}d^e\sum_{j=1}^qG_j(y')T_j)$ from $P(y')+d^e(T)D[Y,T]_s$. Thus $ (({h},{g})D[Y,{T}]_{s}):I$ 
contains an element of type $ds''$ for some $s''\in s+d^{e-1}(T)D[Y,T]_s$ and we get  
  ${s}''IV\subset (0:_{V}d)\cap d^eV=0$ because  $V$ is flat over $D$ and so over $A$ by hypothesis. Then ${E}_{{s}{s}'{s}''}\cong {V}_{{s}''} $ is a $B$-algebra which is also  smooth over $D$.

 As $\omega({s}), {\psi}({s}'),{\psi}({s}'')\equiv 1$ modulo $(d)$  we see that $\omega({s}),{\psi}({s}'), {\psi}({s}'')$ are invertible because  $A'$ is local. Thus ${\psi}$ (and so $v$) factors through the  $A$-algebra $B'={E}_{{s}{s}'{s}''}$. Clearly, $h_D\subset h_{B'}$.
\hfill\  \end{proof}

Next theorem is an extension of Theorem \ref{m}.

 \begin{Theorem} \label{m1}  Let $u:A\to A'$ be a flat morphism of Noetherian  rings and $q\in \Spec A'$. Suppose that   $A'$ is local, $p=u^{-1}(q)\in \Spec A$ generates the maximal ideal $qA'_q$ of $A'_q$ and $u$ induces a regular morphism $A_p\to A'_q$. Any $A$-morphism $v:B\to A'$  such that $v(H_{B/A})A'_q$ is $q A'_q$-primary ideal factors through an  $A$-algebra of finite type  $B'$,  let us say $v$ is the composite map $B\to B'\xrightarrow{w} A'$, with  $h_B\subset h_{B'}=\sqrt{w(H_{B'/A})A'}\not \subset q$, that is $A\to B\to A'\supset q$ is resolvable.
 \end{Theorem}
\begin{proof} Let $m=\height\ p=\height\ q$ and choose $\gamma_1,\ldots,\gamma_m\in u^{-1}(v(H_{B/A}A'_q)$, which induces a system of parameters in $A'_q$. The proof goes with $A'_q$ instead $A'$ as in \cite[Theorem 2]{ZKPP} up to the "Cases", where we proceed as follows.

{\bf Case I} If $m=0$ then $A_p,\ A'_q$ are Artinian local rings and by
 \cite[Corollary 3.3]{P0} we get that $A_p\to A_p\otimes_A B\to A'_q\supset qA'_q$ is resolvable. Then by   Lemma  \ref{l0} it follows that $A\to B\to A'\supset q$ is resolvable.  

{\bf Case II}  If $m>0$  we assume  by induction hypothesis as in \cite[Theorem 2]{ZKPP} $\gamma=\gamma_m$, $d=d_m$ and the existence of 
 an $\bar A=A/(d^{2e+1})$-algebra  of finite type ${\bar D}=D/(d^{2e+1})D\cong({\bar A}[Z]/({\bar g}))_{{\bar h}{\bar M}}$, $Z=(Z_1,\ldots,Z_r)$, ${\bar g}=({\bar g}_1,\ldots,{\bar g}_s)$ such that $\bar{v}=\bar{A}\otimes_Av$ factors through $\bar D$, let us say $\bar v$ is the composite map $\bar{B}\to \bar{D}\xrightarrow{\bar{\omega}} \bar{A}'$ and $h_B{\bar A}'\subset h_{{\bar D}}= \sqrt{{\bar \omega}(H_{\bar{D}/\bar{ A}}){\bar A}'}\not \subset q{\bar A}'$.

Now, let $g, M,h$ be some liftings of $\bar{g}, \bar{M}, \bar{h} $ in $A[Z]$, $f_i=g_i-d^{2e+1}T_i$, $T=(T_1,\ldots,T_s)$ and 
set $D=A[Z,T]/(f)$. Let $t\in A'^s$ be such that $g(z)=d^{2e+1}t$. The map $\omega:D\to A'$ given by $(Z,T)\to (z,t)$ lifts  $\bar{\omega}$. It follows that $\Im v\subset \Im\  \omega +d^{2e+1}A'$. Note that $d\in h_D=\sqrt{\omega(H_{D/A})A'}$. It follows that $h_B\subset h_D\not \subset q$ because ${\bar A}\otimes_AD\cong {\bar D}[T]$ is smooth over $\bar D$ and  $h_B{\bar A}'\subset h_{{\bar D}}\not \subset q{\bar A}'$.

Applying Proposition \ref{p1} we see that $v$ factors through an $A$-algebra  of finite type  $B'$,  let us say $v$ is the composite map $B\to B'\xrightarrow{w} A'$ and  $h_D\subset h_{B'}=\sqrt{w(H_{B'/A})A'}$. It follows that $A\to B\to A'\supset q$ is resolvable.
\hfill\ \end{proof}

\begin{Remark}\label{r} {\em The above proof of the  theorem seems to be not constructive as it is Theorem \ref{m} (more details in Proposition \ref{p2}). If $A'={\bf Q}[[x_1.x_2]]$ and $\height q=1$ then to resolve $A\to B\to A'\supset q$  we might need to define precisely $v$ modulo a power of $q$, that is to give formal power series in $x$  with infinite number of coefficients. 
Indeed, our procedure is based in resolving $A/(d^{2e+1})\to B/d^{2e+1}B\to A'/d^{2e+1}A'\supset q/d^{2e+1}A'$ for some $d\in A$ with $dA'\subset \sqrt{v(H_{B/A})A}'$. As in the proof of  Proposition \ref{p1} we had to find $y'$ in $D$ with $\omega(y')\equiv y$ modulo $d^{2e+1}A'$, or modulo an ideal generated by several powers of such $d$, or almost modulo a high enough power of a minimal prime ideal $q$ associated to $v(H_{B/A})A'$. When $A'/q$ is not Artinian then it is not enough to use a high jet of $y$.
When $q=(x_1,x_2)$ then to define $v$ modulo a power of $q$ means only to give a finite number of such coefficients (this follows from the proof of Theorem \ref{m}).}
\end{Remark}

\begin{Lemma}\label{l}
Let  $u:A\to A'$ be a regular morphism of Noetherian  rings, $B$ an $A$-algebra of finite type, $v:B\to A'$ an $A$-algebra 
morphism and $q\in \Spec A'$ a minimal prime associated ideal of $h_B$. Suppose that $A'$ is local, $\height q\geq 1$ and  char $ A_p/pA_p=0$, $p=u^{-1}(q)$.
Then $A\to B\to A'\supset q$ is resolvable, that is  $v$ factors through an $A$-algebra of finite type $C$, let us say $v$ is the composite map $B\to C\xrightarrow{w} A'$, such that $h_B\subset h_C=\sqrt{w(H_{C/A})A'}\not \subset q$. 
\end{Lemma}
\begin{proof} 
 Apply  Theorem  \ref{m1} if $t:=\height q-\height p=0$. Otherwise,
$A'_q/p A'_q$ is a regular local ring of dimension $t>0$ by the regularity of $u$. Suppose that 
$z=(z_1,\ldots,z_t)\in A'^t$ induces a regular system of parameters in $A'_q/p A'_q$.
Let $A[Z]$ be the polynomial $A$-algebra in the variables $Z=(Z_1,\ldots,Z_t)$. By the local criterion of flatness \cite{M}, the map $u_1:A_1=A_p[Z]_{(p,Z)}\to A'_q$ given by $Z\to z$
is flat, even regular since $A'_q/p A'_q$ is regular and char  $ A_p/pA_p=0$. Let $v_1:B_1=A_1\otimes_AB\to A'_q$ be the extension of $v$ given by $Z\to z$.
Then $v_1(H_{B_1/A_1})A'_q$ is a $q A'_q$-primary ideal and applying Theorem \ref{m1} we see that $A[Z]\to B[Z]\to A'\supset q$ is resolvable and so $A\to B\to A'\supset q$ is resolvable too.
\hfill\ \end{proof}

\begin{Theorem} \label{gnd0}
Let $u:A\to A'$ be a regular morphism of Noetherian  $\mathbb{Q}$-algebras, $B$ an $A$-algebra of finite type and $v:B\to A'$ an $A$-algebra 
morphism. Suppose that $A'$ is a local ring.
Then $v$ factors through a standard smooth  $A$-algebra.
\end{Theorem}

\begin{proof} As in \cite[Corollary 3.3]{P0}  we may suppose that $\height v(H_{B/A})A'\geq 1$. This could be done in  a constructive way (see for instance  the beginning of Section 2 of \cite{AP} in the case when $A'$ is a domain). Let  $q\in \Spec A'$ be a minimal prime associated ideal of $v(H_{B/A})A'$. Applying Lemma \ref{l} we see that $v$ factors through an $A$-algebra of finite type $B_1$, let us say $v$ is the composite map $B\to B_1\xrightarrow{w_1} A'$, such that $v(H_{B/A})A'\subset \sqrt{ w_1(H_{B_1/A})A'}\not \subset q$.  Applying again Lemma \ref{l} for $B_1$ and a certain minimal prime associated ideal $q_1$ of $w_1(H_{B_1/A})A' $ we see that $w_1$ factors through an $A$-algebra of finite type $B_2$, let us say $w_1$ is the composite map $B_1\to B_2\xrightarrow{w_2} A'$, such that $w_1(H_{B_1/A})A'\subset \sqrt{ w_2(H_{B_2/A})A'}\not \subset q_1$. By Noetherianity this procedure should end since the strictly increasing chain  $v(H_{B/A})A'\subset \sqrt{ w_1(H_{B_1/A})A'}\subset \sqrt{w_2(H_{B_2/A})A'}\subset \ldots $ should stop. This could happens only when $w_i(H_{B_i/A})A'=A'$ for some $i$.
Using \cite[Lemma 2.4]{P0} we may replace $B_i$ by an $A$-algebra $C$   
of finite type with $H_{C/A}=C$, that is $C$ is a smooth algebra over $A$. We may choose $C$ to be standard smooth over $A$ using \cite[Lemma 3.4]{P0} (see also \cite[Lemma 4]{ZKPP}).
\hfill\ \end{proof}

Next, let $u:A\to A'$ be a regular morphism of Noetherian  local rings of dimension two, $B$ an $A$-algebra of  finite type and $v:B\to A'$ an A-morphism. If $h_B=\sqrt{v(H_{B/A})A'}$ is the maximal ideal of $A'$ then we have a constructive General Neron Desingularization using Theorem \ref{m} (see also \cite{PP1}). Always we may reduce to the case $\height h_B\geq 1$ as we noticed in the proof of Theorem \ref{gnd0}. The difficulties to construct a General Neron Desingularization could appear only when we show that $A\to B\to A'\supset q $ is resolvable for a minimal prime associated ideal $q$ of $h_B$ with $\height q=1$.

\begin{Proposition}\label{p2} If there exists  a minimal associated ideal $q$ of $h_B$ such that  $u^{-1}(q)$ has height $1$ then $A\to B\to A'\supset q$ is resolvable in a constructive way. Consequently, in this case Theorem \ref{gnd0} holds in a constructive way.
\end{Proposition} 

\begin{proof} Let  $B=A[Y]/I$, $Y=(Y_1,\ldots,Y_n)$ and $q$ be a minimal associated ideal  of $h_B$ such that $u^{-1}(q)$ has height $1$. Choose  $\gamma\in u^{-1}(v(H_{B/A})A')$ such that $\height (\gamma A')=1$.  As in the proof of \cite[Theorem 2]{ZKPP} we may change $B$ such that a power $d$ of $\gamma$  is in $((f):I)\Delta_f$ for some polynomials $f=(f_1,\ldots,f_r)$, $r\leq n$ from $I$. Choosing $e$ as in Proposition \ref{p1} it is enough to see  as in {\bf Case
2} from Theorem \ref{m1} that ${\bar A}=A/(d^{2e+1})\to \bar{B}={\bar A}\otimes_AB\to {\bar A}'=A'/d^{2e+1}A'\supset q{\bar A}'$ is resolvable. But this is done in \cite{PP} because we are now in dimension one.  
\hfill\ \end{proof}

\end{document}